\documentclass{amsart}
\usepackage{amsfonts,amssymb,amscd,amsmath,enumerate,verbatim,calc}
\usepackage[all]{xy}
\newcommand{\CMS}{\operatorname{\underline{CM}}}

\newcommand{\CM}{Cohen-Macaulay}

\newcommand{\B}{\mathcal{B} }

\newcommand{\n}{\mathfrak{n} }
\newcommand{\m}{\mathfrak{m} }

\newcommand{\Z}{\mathbb{Z} }
\newcommand{\V}{Var }
\newcommand{\C}{\mathcal{C} }

\newcommand{\D}{\mathcal{D} }

\newcommand{\T}{\mathcal{T} }
\newcommand{\A}{\mathcal{A} }
\newcommand{\rt}{\rightarrow}

\newcommand{\ov}{\overline}

\newcommand{\bx}{\mathbf{x}}

\newcommand{\wh}{\widehat }

\newcommand{\image}{\operatorname{image}}
\newcommand{\Xb}{\mathbf{X}_\bullet}
\newcommand{\Zb}{\mathbf{Z}_\bullet}
\newcommand{\Yb}{\mathbf{Y}_\bullet}
\newcommand{\Cb}{\mathbf{C}_\bullet}
\newcommand{\Pb}{\mathbf{P}_\bullet}
\newcommand{\mmod}{\operatorname{mod}}

\newcommand{\depth}{\operatorname{depth}}

\newcommand{\cone}{\operatorname{cone}}

\newcommand{\projdim}{\operatorname{projdim}}
\newcommand{\proj}{\operatorname{proj}}
\newcommand{\embdim}{\operatorname{embdim}}

\newcommand{\gr}{\operatorname{gr}}

\newcommand{\curv}{\operatorname{curv}}
\newcommand{\Hom}{\operatorname{Hom}}

\newcommand{\Tor}{\operatorname{Tor}}

\theoremstyle{plain}

\newtheorem{theorem}{Theorem}[section]
\newtheorem{corollary}[theorem]{Corollary}
\newtheorem{lemma}[theorem]{Lemma}
\newtheorem{proposition}[theorem]{Proposition}

\theoremstyle{definition}

\newtheorem{s}[theorem]{}
\newtheorem{remark}[theorem]{Remark}
\newtheorem{example}[theorem]{Example}

\theoremstyle{remark}

\begin{document}

\title[Bounded betti numbers]{Lengths of modules  over short Artin local rings }
\author{Tony~J.~Puthenpurakal}
\date{\today}
\address{Department of Mathematics, IIT Bombay, Powai, Mumbai 400 076}

\email{tputhen@math.iitb.ac.in}

 \begin{abstract}
Let $(A,\m)$ be a short Artin local ring (i.e., $\m^3 = 0$ and $\m^2 \neq 0$). Assume $A$ is not  a hypersurface  ring. We show  there exists $c_A \geq 2$ such that if $M$ is any finitely generated module with bounded betti-numbers then $c_A $ divides $\ell(M)$, the length of $M$.  If $A$ is not a complete intersection then there exists $b_A \geq 2$ such that if $M$ is any module with $\curv(M) < \curv(k)$ then $b_A$ divides $\ell(\Omega^i_A(M))$ for all $i \geq 1$ (here $\Omega^i_A(M)$ denotes the $i^{th}$-syzygy of $M$).
\end{abstract}
 \maketitle
\section{introduction}
Let $(A,\m)$ be an Artin local commutative ring with residue field $k$. All modules under discussion are finitely generated. Let $M$ be an $A$-module.  Let $\beta_n(M) = \ell(\Tor^A_n(M,k))$ be the $n^{th}$-betti number of $M$ (here $\ell(-)$ denotes length). In general the sequence $\{ \beta_n(M) \}_{ n \geq 0}$ might be unbounded. It is thus interesting to study non-free modules with bounded betti numbers.
Let $P_M(z) = \sum_{n \geq 0} \beta_n(M)z^n$ be the Poincare series of $M$. If $M =k$ then we by abuse of notation call it the Poincare series of $A$.

 If $A$ is a hypersurface ring i.e.,
$A = Q/(f)$ where $(Q,\n)$ is regular and $f\in \n^2$, then it is well-known that every module $M$ has bounded betti-numbers (in fact it is eventually periodic), see \cite[6.1]{E}.
So we will assume that $A$ is not a hypersurface ring.

 If $\m^2 = 0$ and $\ell(\m) \geq 2$ then every non-free module $M$ has the property that its first syzygy is $\cong k^r$ for some $r \geq 1$. Thus $M$ has unbounded betti-numbers, for instance see \ref{min-mult}.  Thus the simplest non-trivial case is when $A$ is short, i.e.,  $\m^3 = 0$ and $\m^2 \neq 0$. Note that such rings can exhibit sufficiently general behaviour. For instance the Poincare series of any local ring $R$ is rationally related to the Poincare series of a short Artin local ring, see \cite[Theorem 2]{AG}. It was conjectured by Eisenbud, \cite{E}, that if $M$ has bounded betti numbers then $M$ is eventually periodic (with period $\leq 2$). This conjecture  was disproved by Gasharov and Peeva \cite{GP}. In fact they construct  short Artin local rings disproving Eisenbud conjecture. They   give examples of short Artin rings having periodic module of period any integer $n \geq 2$. They  also gave an example of a short Artin ring having a non-free module $M$ with bounded betti-numbers such that $M$ is not eventually periodic.  See \cite[3.4]{GP} for these examples. Furthermore short Artin  rings have been source of many examples/counterexamples in local algebra, see \cite{A}, \cite{R}, \cite{JS} and \cite{CV}.

\s  Our first result is
 \begin{theorem}\label{main}
   Let $(A,\m)$ be a short Artin local ring. Assume $A$ is not  a hypersurface  ring. Then there  exists $c_A \geq 2$ such that if $M$ is any finitely generated module with bounded betti-numbers then $c_A $ divides $\ell(M)$.
 \end{theorem}
The proof of Theorem \ref{main} splits into two cases.

The first case is when $A$ is a complete intersection of length four. For complete intersections there is a notion of support varieties, see \cite{AB}.  We note that if $k = A/\m$ is algebraically closed then for each point $\{ a \}$ in $\mathbb{P}^1_k$  there exists a module $M$ with bounded betti numbers and $\V(M)  = \{ a \}$. Furthermore if $A$ is equi-characteristic then  for each $a \in \mathbb{P}^1_k$ there exists indecomposable modules $\{ M_n \}_{n \geq 1}$ with $\V(M_n) = \{ a \}$ and  $\ell(M_n) \geq n$, see \cite[1.6]{P} (the proof of \cite[1.6]{P} also essentially works in mixed characteristic (for the Artinian case), one has to use \cite[VI.1.4]{ARS} and arguments similar to proof of \cite[1.6]{P}.)

The second case is when $A$ is not a complete intersection. In this case if $A$ has a non-free module with bounded betti numbers then by \cite[Theorem B]{L}
we have $\ell(\m/\m^2) = h + 1$, $\ell(\m^2) = h \geq 2$, socle of $A$ is $\m^2$ and its Poincare series is
\[
P_A(z) = \frac{1}{1- (h+1)z + hz^2} = \frac{1}{(1-z)(1-hz)}.
\]
We note that the rings in  examples in \cite{GP}, \cite{JS} and \cite{CV} are of this form.

\s There is no general method to construct modules with bounded betti-numbers. If $A = R/(f)$ where $R$ is local of dimension one and $f \in \n^2$ is regular then let $\mathcal{P}_R$ be the full subcategory of $A$-modules $M$  with $\projdim_R M $ finite. Such modules if not free have periodic resolution (over $A$)  with period $\leq 2$. There is essentially a unique method to construct non-free
modules over $A$ having finite projective dimension over $R$. This is
essentially due to Buchweitz et al, see  \cite[2.3]{BGS}. Also see the paper \cite[1.2]{BHU} by Brennan et al.

\s We say an Artin local ring $R$ has  property $\B$ if there exists $c \geq 2$ such that $c$ divides $\ell(M)$ for every $R$-module $M$ with bounded betti-numbers. By Theorem \ref{main}, short Artin local
rings satisfy $\B$. The next result shows that property $\B$ is preserved under certain flat extensions.
\begin{proposition}\label{flat}
Let $(R,\m) \rt (S,\n)$ be an extension of Artin local rings such that $S$ is a finite free $R$-module. Assume the induced extension of residue fields $R/\m \rt S/\n$ is an isomorphism. Then if $R$ has property $\B$ then so does $S$
\end{proposition}
Proposition \ref{flat} produces bountiful examples of Artin local rings which satisfy $\B$. See section \ref{flat-section}. The following is a consequence of Theorem \ref{main} and Proposition \ref{flat},
\begin{corollary} \label{ci-cor}
Let $A = k[X_1,\ldots, X_d]/(X^{a_1}, \ldots, X^{a_d})$ where $d \geq 2$ and $a_i \geq 2$. If two among $a_i$ are even then if $M$ is any $A$-module with bounded multiplicity then its length is even.
\end{corollary}
\s  For our next result let us recall the notion of curvature of a module
 \[
 \curv(M) = \limsup_{n \rt \infty} \sqrt[n]{\beta_n(M)}.
 \]
 It is known that $\curv(M) \leq \curv(k) < \infty$ (see \cite[4.2.4]{A}. Furthermore if $A$ is not a complete intersection then $\curv(k) > 1$, see \cite[8.2.2]{A}.
 Our next result is
 \begin{theorem}\label{curv}
 Let $(A,\m)$ be a short Artin local ring. Assume $A$ is not  a complete intersection  ring. There exists $b_A \geq 2$ such that if $M$ is any module with $\curv(M) < \curv(k)$ then $b_A$ divides $\ell(\Omega^i_A(M))$ for all $i \geq 1$.
 \end{theorem}
 By \cite[Theorem B]{L} if a short ring has a module $M$ with $\curv(M) < \curv(k)$ then its Poincare series is

 \[
P_A(z) = \frac{1}{1- dz + az^2} = \frac{1}{(1-r_1z)(1-r_2z)}.
\]
where $d = \ell(\m/\m^2)$, $a = \ell(\m^2)$,  socle of $A$ is $\m^2$, $r_1$ and $r_2$ are positive integers with $r_1 < r_2$ and $\curv(M) = r_1 $ and $\curv(k) = r_2$. Furthermore $a \geq d -1$.

\s \emph{A result for a higher dimensional ring.}\\
Let $(R,\m)$ be a \CM \ local ring of dimension $d  \geq 1$. For convenience we will assume that the residue field $k =A/\m$ of $A$ is infinite.
Set $h = \embdim(R) - d$ and let $e_0(R)$ be its multiplicity.
Let $\gr R = \bigoplus_{n \geq 0}\m^n/\m^{n+1}$ be the associated graded ring of $R$.
Recall $R$ is said to have minimal multiplicity if  $e_0(R)  =  h +1$. We assume $R$ is not regular.
It is well-known that if $R$ has minimal multiplicity then  $\gr R$ is \CM, see \cite{S-min}. Let $g \in \m^2\setminus \m^3$ be such that the initial for $g^*$ of $g$ in $\gr R$ is a non-zero divisor.
Let
 $$\B_g =  \{ M \mid M \ \text{ is a MCM $R/(g)$-module such that $\projdim_R M <  \infty$}\}.$$
Here MCM stands for maximal \CM. Also set
 $$ \B^{(2)}  = \bigcup_{ \stackrel{\deg g^* = 2}{ g^* \text{nzd in} \ \gr R}}\B_g. $$
 We prove
 \begin{theorem}\label{min}
   (with hypotheses as above). If $M \in \B^{(2)}$ then $ h + 1$ divides $e_0(M)$.
 \end{theorem}

\s \emph{Technique used to prove our results:} \\
The techniques to prove our results to the best of our knowledge have not been used  earlier in commutative algebra.
We first discuss proof of Theorem \ref{main} when $A$ is a  Artin complete intersection of length $4$. Let $\CMS(A)$ denote the stable category of $A$-modules and let $\CMS^{\leq 1}(A)$ denote the category of modules with bounded betti numbers. Then $\CMS^{\leq 1}(A)$ is a thick subcategory of $\CMS(A)$ and let $\T$ denote the Verdier quotient $\CMS(A)/\CMS^{\leq 1}(A)$.
We have an exact sequence of Grothendieck groups
\[
G_0(\CMS(A)) \xrightarrow{\xi} G_0(A) \rt G_0(\T) \rt 0.
\]
It suffices to show $\xi$ is NOT surjective. Equivalently we have to prove $G_0(\T) \neq 0$. We recall a result due to Thomason \cite{T}. Let $\C$ be a skeletally small triangulated category. Recall a subcategory $\D$ is dense in $\C$ if the smallest thick subcategory of $\C$ containing $\D$ is $\C$ itself. In \cite{T} a one-to one correspondence between dense subcategories of $\C$ and subgroups of the Grothendieck group $G_0(\C)$ is given. In particular if $G_0(\C) = 0$ then any dense subcategory of $\C$ is $\C$ itself. Now $G_0(\CMS(A)) = \Z/4\Z$. Let $\A$ be the dense subcategory of $\CMS(A)$ corresponding to the subgroup $\{ 0\}$ of $G_0(\CMS(A))$. By a general construction in \ref{dense} we construct a dense subcategory $\A^*$ in $\T$ which is closely related to $\A$. So if $G_0(\T) = 0$ then $\A^* = \T$. In particular $k \in \A^*$. We use this fact to get a contradiction, see \ref{ci}.

The technique to prove Theorem \ref{main} (when $A$ is not a complete intersection) and Theorem \ref{curv} is similar. We have to work with $D^b(A)$ the bounded derived category of $A$ and construct an appropriate thick subcategory $\C$ of $D^b(A)$. The proofs consists of analyzing $G_0(D^b(A)/\C)$ and showing either  it is non-zero or if it zero then it still yields some information on $\ell(M)$ (for suitable $M$).

To prove Theorem \ref{min} we use techniques used in the proof of Theorem \ref{main}.

We now describe in brief the contents of this paper. In section two we discuss some preliminaries on triangulated categories, Grothendieck groups and Hilbert functions that we need. In section three we describe a dense category of a Verdier quotient  of $\C$ corresponding to a dense sub category of $\C$ (here $\C$ is a triangulated category). In section four we prove Theorem \ref{main} when $A$ is a Artin complete intersection of multiplicity four.  In section five we describe a construction which is useful to us. In section six we prove Theorem \ref{main}. In section seven we prove Proposition \ref{flat} and give bountiful examples of rings satisfying property $\B$. In the next section we prove Theorem \ref{curv}.  In section nine we prove Theorem \ref{min}. In section ten we  prove some results on ration of betti-numbers that we need. The results of this section are essentially known. We provide details due to lack of a suitable reference. In the appendix we calculate a limit which is crucial for us.

\section{Preliminaries}
Throughout this paper all rings are Noetherian and all modules considered are finitely generated.

\s \emph{Triangulated categories:}\\
We use \cite{N} for notation on triangulated categories. However we will assume that if $\C$ is a triangulated category then $\Hom_\C(X, Y)$ is a set for any objects $X, Y$ of $\C$.

 \s \label{sub-tri} Let $\C$ be a triangulated category with shift functor $\sum$. A full subcategory $\D$ of $\C$ is called a \emph{triangulated subcategory} of $\C$ if
 \begin{enumerate}
   \item  $X \in \D$ then $\sum X, \sum^{-1}X\in \D$.
   \item If $X \rt Y \rt Z \rt \sum X$ is a triangle in $\C$ then if $X, Y \in \C$ then so does $Z$.
   \item If $X \cong Y$ in $\C$ and $Y \in \D$ then $X \in \D$.
 \end{enumerate}
 \begin{remark}
   In some sources a triangulated category is defined to have only property (1) and (2). Furthermore in this notation if a triangulated sub category satisfies (3) then it is called a strict triangulated subcategory. However for us triangulated subcategories is as in \ref{sub-tri}. This is also as in \cite{N}.
 \end{remark}

 \s A triangulated subcategory $\D$ of $\C$ is said to be \emph{thick} if $U \oplus V \in \D$ then $U, V \in \D$.
  A triangulated subcategory $\D$ of $\C$ is \emph{dense} if for any $U \in \C$ there exists $V \in \C$  such that $U \oplus V  \in \D$.

 \s\label{Groth} Let $\C$ be a skeletally small triangulated category. The \emph{Grothendieck group}  $G_0(\C)$ is the quotient group of the free abelian group
on the set of isomorphism classes of objects of $\C$ by the Euler relations: $[Y]  = [X] + [Z] $ whenever $X \rt Y \rt Z \rt \sum X$ is a triangle in $\C$.
We always have a triangle $X \rt 0 \rt \sum X \xrightarrow{1} \sum X$. So we have $[X] + [\sum X] = [0] = 0 $ in $G_0(\C)$. Therefore $[\sum X] = - [X]$ in $G_0(\C)$. It follows that \emph{every}
element of $G_0(\C)$ is of the form $[X]$ for some $X \in \C$.

\s Let $\C$ be skeletally small and let $\D$ be a thick subcategory of $\C$. Set $\T = \C/\D$  be the Verdier quotient. Then there exists an exact sequence (see \cite[VIII, 3.1]{I}, also see
\cite[II.6.4]{W})
$$ G_0(\D) \rt G_0(\C) \rt G_0(\T) \rt 0. $$

\s (with setup as in \ref{Groth}). Thomason \cite[2.1]{T} constructs a one-to-one correspondence between dense subcategories of $\C$ and subgroups of $G_0(\C)$ as follows:

To $\D$ a dense subcategory of $\C$ corresponds the subgroup which is the image of $G_0(\D)$ in $G_0(\C)$. To  $H$ a subgroup of $G_0(\C)$
corresponds the full subcategory $\D_H$ whose objects are those $X$ in $\C$ such that
$[X] \in H$.

\s Let $D^b(A)$ be the bounded derived category of a  ring $A$. We index complexes cohomologically.  The obvious functor $\mmod(A) \rt D^b(A)$ yields a group homomorphism $\phi \colon G_0(A) \rt G_0(D^b(A))$. The map $\phi$ is an isomorphism with inverse $\psi \colon G_0(D^b(A)) \rt G_0(A)$ defined by $\psi(\Xb)  = \sum_{i \in \Z} (-1)^i[H^i(\Xb)]$.

\s Let $(A,\m)$ be an Artin local ring. Then we have a group isomorphism $f \colon G_0(A) \rt \Z$ defined by $f([M]) = \ell(M)$.

\s Let $(A,\m)$ be Artinian Gorenstein local ring. Let $\CMS(A)$ be the stable category of $A$-modules. Then $G_0(\CMS(A)) = \Z/\ell(A)\Z$, see \cite[4.4.9]{Buch}
\s \emph{Associated graded rings, modules, Hilbert functions, superficial elements and multiplicity.}\\
Let $(A, \m)$ be local. Let $\gr A = \bigoplus_{n \geq 0}\m^n/\m^{n+1}$ be the \emph{associated graded ring} of $A$.  We note that $\gr A$ is a graded Noetherian $k = A/\m$-algebra.
If $a \in A$ is non-zero then $a \in \m^i \setminus \m^{i+1}$ for some $i$. Then $a^* = $ image of $a$ in $\m^i/\m^{i+1}$ (considered as a subset of $\gr A$) is called the initial form of $a$.
Let $M$ be an $A$-module. Let  $\gr M = \bigoplus_{n \geq 0}\m^nM/\m^{n+1}M$ be the \emph{associated grade module} of $M$. Note $\gr M$ is finitely generated as a $\gr A$-module.

\s The function $H(M, n) = \ell(\m^n M/\m^{n+1} M)$  for $n \geq 0$ is called the Hilbert function of $M$. We assemble it $H_M(z) = \sum_{n \geq 0}H(M, n)z^n$, the Hilbert series of $M$. It is well-known that
\[
H_M(z) = \frac{h_M(z)}{(1-z)^{\dim M}}, \quad \text{where} \ h_M(z) \in \Z[z] \ \text{and} \ h_M(1) \neq 0.
\]
The element $h_M(1) = e_0(M)$ is called the multiplicity of $M$.

\s \label{reg-hilb} If $a \in \m^r \setminus \m^{r+1} $ (here $r \geq 1$) is such that $a^*$ is $\gr M$-regular then $a$ is $M$-regular. Set $N = M/aM$. Then $e_0(N) = r e_0(M)$. Furthermore
$$ h_N(z) = h_M(z) ( 1 + z + \cdots + z^{r-1}).$$

\s An element $x \in \m$ is said to be $M$-superficial if there exists $c $ and $n_0$ such that for all $n \geq n_0$ we have $(\m^{n+1}M \colon x)\cap \m^c M = \m^{n}M$.
Superficial elements exist if $k$ is infinite. A sequence $x_1, \ldots, x_s$ is called an $M$-superficial sequence if $x_i$ is $M/(x_1, \ldots, x_{i-1})$-superficial for $i = 1, \ldots, s$.

\s If $\depth M > 0$ then every $M$-superficial element $x$ is $M$-regular. Furthermore in this case we have $(\m^{n+1}M \colon x)  = \m^n M$ for $n \gg 0$.

\s \label{sup-reg} Let $x_1, \ldots, x_r$ be a $M$-superficial sequence. Then $\depth \gr M \geq r$ if and only if $x_1^*, \ldots, x_r^*$ is $\gr M$-regular, see \cite[Theorem 8]{P-th}.

\s\label{mult-sup} Suppose $\depth M \geq r$ and $x_1, \ldots, x_r$ is a $M$-superficial sequence. Set $N = M/(x_1, \ldots, x_r)M$. Then $e_0(N) = e_0(M)$, see \cite[Corollary 11]{P-th}.
\section{A subcategory associated to a Verdier quotient}
In this section $\C$ is a  triangulated category with shift functor $[1]$, $\T$ is a thick subcategory of $\C$ and $\D = \C/\T$ is the Verdier quotient.
Let $\A$ be a dense triangulated subcategory of $\A$.
Consider the full subcategory $\A^*$ of $\D$ whose objects are
$$\{ X \mid X \cong Y \ \text{in $\D$ for some $Y$ in $\A$} \}.$$
In this section we prove
\begin{theorem}\label{dense}
(with hypotheses as above) $\A^*$ is a dense triangulated sub-category of $\D$.
\end{theorem}
\begin{proof}
It is clear that if $X \in \A^*$ then $X[1], X[-1] \in \A^*$.
It is clear that $\A^*$ is dense.
It remains to show $\A^*$ is triangulated.
Let
$t' \colon X' \rt Y' \rt Z' \rt X'[1]$ be a triangle in $\D$ with $X', Y' \in \A^*$. Then $t'$ is the image in $\D$ of a triangle
$$t \colon X \xrightarrow{\eta} Y \rt Z \rt X[1] \quad \text{ in $\C$.} $$
We note $ \phi \colon X \cong \wh{X}$ in $\D$ with $\wh{X} \in \A$. We have left fraction
\[
\xymatrix{
\
&U
\ar@{->}[dl]_{u}
\ar@{->}[dr]^{f}
 \\
X
\ar@{->}[rr]_{\phi = fu^{-1}}
&\
&\wh{X}
}
\]
where $\cone(u) \in \T$. As $\phi$ is an isomorphism we also have $\cone(f) \in \T$.
We have a morphism of triangles
\[
  \xymatrix
{
s \colon
 &U
\ar@{->}[r]^{\eta \circ u}
\ar@{->}[d]^{u}
 & Y
\ar@{->}[r]
\ar@{->}[d]^{1_Y}
& \cone(\eta\circ u)
\ar@{->}[r]
\ar@{->}[d]
& U[1]
\ar@{->}[d]
\\
t \colon
& X
\ar@{->}[r]^{\eta}
 & Y
\ar@{->}[r]
& Z
    \ar@{->}[r]
    &X[1]
\
 }
\]
Note in $\D$ the above morphism of triangles is an isomorphism. Set $W = \cone(\eta \circ u)$. Consider the morphism of triangles
\[
  \xymatrix
{
\widetilde{s} \colon
 & W[-1]
\ar@{->}[r]^{v}
\ar@{->}[d]^{1_{W[-1]}}
 & U
\ar@{->}[r]
\ar@{->}[d]^{f}
& Y
\ar@{->}[r]
\ar@{->}[d]
& W
\ar@{->}[d]
\\
r \colon
& W[-1]
\ar@{->}[r]^{f \circ v}
 & \wh{X}
\ar@{->}[r]
& \cone(f\circ v)
    \ar@{->}[r]
    & W
\
 }
\]
Here $\widetilde{s}$ is a rotation of $s$. Note $r$ is isomorphic to $\widetilde{s}$ in $\D$. Set $L = \cone(f \circ v)$. Rotating $r$ we get a triangle
$$ \widetilde{r} \colon \wh{X} \xrightarrow{\delta} L \rt W \rt \wh{X}[1] $$
with $\widetilde{r} \cong t$ in $\D$. However note that $\wh{X} \in \A$. We have $ \psi\colon L \cong \wh{L}$ in $\D$ where $\wh{L} \in \A$.
Consider the right fraction
\[
\xymatrix{
\
&T
\ar@{<-}[dl]_{g}
\ar@{<-}[dr]^{w}
 \\
L
\ar@{->}[rr]_{\psi = w^{-1}g}
&\
&\wh{L}
}
\]
where $\cone(w) \in \T$. As $\psi$ is an isomorphism we also have $\cone(g) \in \T$.
We have a morphism of triangles
\[
  \xymatrix
{
\widetilde{r} \colon
 & \wh{X}
\ar@{->}[r]^{\delta}
\ar@{->}[d]^{1_{\wh{X}}}
 & L
\ar@{->}[r]
\ar@{->}[d]^{g}
& W
\ar@{->}[r]
\ar@{->}[d]
& \wh{X}[1]
\ar@{->}[d]
\\
l \colon
& \wh{X}
\ar@{->}[r]^{g \circ \delta}
 & T
\ar@{->}[r]
& \cone(g\circ \delta)
    \ar@{->}[r]
    & \wh{X}[1]
\
 }
\]
We note that $ t\cong \widetilde{r} \cong l $ in $\D$.
We have  a triangle
$$ \wh{L} \xrightarrow{w} T \rt C \rt \wh{L}[1] \ \quad \text{where $C = \cone(w) \in \T$.}$$
As $\A$ is dense in $\C$ we get that $C\oplus C[1] \in \A$, see \cite[4.5.12]{N}.
Taking the direct sum of the above triangle with the triangle $h \colon 0 \rt C[1] \xrightarrow{1} C[1] \rt 0$ we obtain that $T\oplus C[1] \in \A$.
We take the direct sum of $h$ and $l$ to obtain the triangle
$$\widetilde{l} \colon  \wh{X} \rt T\oplus C[1] \rt \cone(g\circ \delta)\oplus C[1] \rt \wh{X}[1].$$
Note  $\cone(g\circ \delta)\oplus C[1] \in \A$. As $h = 0$ in $\D$ it follows that $\widetilde{l} \cong l \cong t$ in $\D$. The result follows.
\end{proof}

\section{Artin Complete intersections of multiplicity four}
In this section $(A,\m)$ is a  Artin local complete intersection of multiplicity ($=$ length) four. It is readily seen that in this case $A = Q/(f,g)$ where $(Q,\n)$ is regular local and $f,g \in \n^2 \setminus\n^3$ is a regular sequence.

Let $\CMS(A)$ denote the stable category of finitely generated $A$-modules (these are maximal \CM \ as $A$ is Artinian).
Note $\CMS(A)$ is a triangulated category, see \cite[4.7]{Buch}.
Let $\CMS^{\leq 1}(A)$ be the full subcategory of $A$-modules with complexity $\leq 1$ (equivalently the category of all $A$-modules with bounded betti numbers). It is readily seen that $\CMS^{\leq 1}(A)$ is a thick subcategory of $\CMS(A)$. Let $\T = \CMS(A)/\CMS^{\leq 1}(A)$ be the Verdier quotient.

  We note the map
\begin{align*}
  \eta   \colon G_0(\CMS(A)) &\rt \Z/4\Z, \\
   M  &\mapsto \ell(M) + 4 \Z.
\end{align*}
is an isomorphism.  We have an exact sequence
\[
G_0(\CMS^{\leq 1}(A)) \xrightarrow{\xi_A} G_0(\CMS(A)) \rt G_0(\T) \rt 0.
\]
We show
\begin{theorem}\label{ci}(with hypotheses as above.)
  Then the map $\xi_A$ is NOT surjective. Equivalently $G_0(\T) \neq 0$. In particular if $M$ is any module of complexity $\leq 1$ then $\ell(M)$ is even.
\end{theorem}
\begin{proof}
 Fisrst  assume that the residue field $k = A/\m$ is algebraically closed. Suppose if possible $G_0(\T) = 0$. Let $\A$ be Thomason dense subcategory corresponding to the zero subgroup of $G_0(\CMS(A))$. Consider the full subcategory $\A^*$  of $\T$ whose objects are
$$\{ X \mid X \cong Y \ \text{in $\T$ for some $Y$ in $\A$} \}.$$
Then by Theorem \ref{dense}, $\A^*$ is a dense triangulated subcategory of $\T$. But $G_0(\T) = 0$. It follows that $\A^* = \T$. So there exists $X \in \A$ such that $X \cong k$ in $\T$.
As $k \neq 0 $ in $\T$ we get that $X$ has complexity two. So there exists $n_0$ such that $\beta_n(X) < \beta_{n+1}(X)$ for all $n \geq n_0$, see \cite[9.2.1]{A}.
Note $\Omega^j_A X \cong \Omega^j_A k$ in $\T$ for all $j \geq 0$. Fix $n \geq n_0$ and let $i \in \{ n, n+1 \}$.
Let $\phi_i \colon \Omega_A^i(X) \cong \Omega_A^i(k)$ be an isomorphism in $\T$.
We have left fraction
\[
\xymatrix{
\
&L_i
\ar@{->}[dl]_{u_i}
\ar@{->}[dr]^{f_i}
 \\
\Omega_A^i(X)
\ar@{->}[rr]_{\phi_i = f_iu_i^{-1}}
&\
&\Omega_A^i(k)
}
\]
where $\cone(u_i) \in \CMS^{\leq 1}(A)$.  As $\phi_i$ is an isomorphism we also get $\cone(f_i) \in \CMS^{\leq 1}(A)$.  Notice
\[
W = \bigcup_{i = n, n+1} \V(\cone(u_i) \cup \V(\cone(f_i))  \quad \text{is a finite set of points in $\mathbb{P}^1_k$}.
\]
Choose a point $a \in \mathbb{P}^1_k \setminus W$. By \cite[2.3]{B}  there exists a module $H$ with $\V(H) = \{ a \}$. It follows that $H$ has complexity one. Furthermore by \cite[6.1, 4.9]{AB} we get
for $j \geq 1$ we have
$$ \Tor^A_j(H, \cone(u_i)) = \Tor^A_j(H, \cone(f_i)) = 0 \quad \text{for $i = n, n+1$}. $$
We may assume that $H$ has no free summands.

We have a short exact sequence of $A$-modules
$$ 0 \rt L_i \rt F_i \oplus \Omega_A^i(k) \rt \cone(f_i) \rt 0, \quad \text{where $F_i$ is a free $A$-module}.$$
It follows that
$$\Tor^A_1(L_i, H)  \cong \Tor^A_1(\Omega_A^i(k), H). $$
Similarly we obtain
$$\Tor^A_1(L_i, H)  \cong \Tor^A_1(\Omega_A^i(X), H). $$
Let $c = \mu(H)$. As $\Tor^A_1(\Omega_A^i(k), H) = k^c$ we get $\Tor^A_1(\Omega_A^i(X), H) = k^c$.
We have an exact sequence
$$ 0 \rt \Omega^1_A(H) \rt A^c \rt H \rt 0.$$
So we have an exact sequence
$$ 0 \rt \Tor^A_1(\Omega^i_A(X), H) \rt \Omega^1_A(H)\otimes \Omega^i_A(X)  \rt \Omega^i_A(X)^c \rt H\otimes \Omega^i_A(X) \rt 0.$$
As $\ell(U\otimes V) \geq \mu(U)\mu(V)$ we obtain
$$ c + c\ell(\Omega^i_A(X)) \geq 2c\beta_i(X).$$
Thus $\ell(\Omega^i_A(X)) \geq 2\beta_i(X) - 1$. But $\Omega^i_A(X) \in \A$. So its length is divisible by four. In particular it is even.
So $\ell(\Omega^i_A(X)) \geq 2\beta_i(X)$ for $i = n, n+1$.
We have a short exact sequence
$$ 0 \rt \Omega^{n+1}_A(X) \rt A^{\beta_n(X)} \rt \Omega^n_A(X) \rt 0.$$
So we get
\begin{align*}
  4\beta_n(X) &= \ell(A) \beta_n(X), \\
   &= \ell(\Omega^n_A(X)) + \ell(\Omega^{n+1}_A(X)), \\
   &\geq 2\beta_n(X) + 2\beta_{n+1}(X), \\
   &> 4 \beta_n(X).
\end{align*}
So we get $1 > 1$, a contradiction. Thus $G_0(\T) \neq 0$.

Now assume $k$ is not algebraically closed. Suppose if possible $\xi_A$ is surjective. Then there exists module $M$ of complexity one with $\ell(M) \cong 1 \mod (4).$
By \cite[App. Th\'{e}or\'{e}me 1, Corollaire]{BourIX},  there exists a flat local extension
$A \subseteq \widetilde{A}$ such that $\widetilde{\m} = \m \widetilde{A}$ is the maximal ideal of $\widetilde{A}$ and the residue field $\widetilde{k}$ of $\widetilde{A}$ is an
algebraically closed extension of $k$. The module $\widetilde{M} = M \otimes_A \widetilde{A}$ has complexity one and its length as an $\widetilde{A}$-module is equal to length of $M$ as an $A$-module. It follows that $\xi_{\widetilde{A}}$ is surjective, a contradiction. Thus $\xi_A$ is not surjective.

Note $\image \xi = (2) \Z/4\Z$ or is zero. Thus if $M$ is any module of complexity $\leq 1$ then either $2$ divides $\ell(M)$ or $4$ does. Either case $2$ divides $\ell(M)$.
\end{proof}

\section{A construction}\label{const}
In this section we make a construction which is useful to us.
\s \label{setup-C} Let $(A,\m)$ be a  Artin local ring. We assume
$$\lim_{n \rt \infty} \beta_{n+1}^A(k)/\beta_n^A(k)  = \alpha > 1.$$
In particular $A$ is not a complete intersection. Let $D^b(A)$ be the bounded derived category of $A$. We often identify $D^b(A)$ with $K^{-,b}(\proj A)$.
In this section we make a construction which is very useful for us.

\s Let $\Xb \in D^b(A)$. We assume $\Xb \in K^{b,-}(\proj A)$.  Let $$\beta_n(\Xb) = \ell(\Hom_{D^b(A)}(\Xb, k[n])).$$ Note if $\Xb$ is a minimal complex then $\beta_n(\Xb)$ is the rank of the free $A$-module $\Xb^{-n}$. Set
$$ \curv(\Xb) = \limsup_{n \rt \infty} \sqrt[n]{\beta_n(\Xb)}. $$
It is not difficult to prove $\curv(\Xb) \leq \alpha$.

\s \label{less} Let $1 < \beta \leq \alpha$. Let $\C_{\beta} = \{ \Xb \mid \curv(\Xb) < \beta \}$. It is clear that $\C_\beta$ is a full subcategory of $D^b(A)$ closed under $[1]$. We show
\begin{lemma}\label{closed}
$\C_\beta$ is a thick triangulated sub-category of $D^b(A)$.
\end{lemma}
\begin{proof}
If $\Xb \in \C_\beta$ then clearly $\Xb[1], \Xb[-1] \in \C_\beta$.
Let $\Xb \rt \Yb \rt \Zb \rt \Xb[1]$ be a triangle with $\Xb, \Yb \in \C_\beta$.We can choose $\gamma < \beta$ such that $\curv(\Xb[1]), \curv(\Yb) < \gamma$. Choose $\epsilon  > 0$ with $\gamma + \epsilon < \beta$. So we have for $n \gg 0$,
$$ \ell(\Hom_{D^b(A)}(\Xb[1], k[n])) < (\gamma + \epsilon)^n \quad \text{and} \quad \Hom_{D^b(A)}(\Yb, k[n])) < (\gamma + \epsilon)^n. $$
It follows that
$$\limsup_{n \rt \infty} \sqrt[n]{ \Hom_{D^b(A)}(\Zb, k[n]))} \leq (\limsup_{n \rt \infty} 2^{1/n} )(\gamma + \epsilon)  = \gamma + \epsilon < \beta. $$
So $\C_\beta$ is a triangulated subcategory of $D^b(A)$. It is elementary to see that it is thick.
\end{proof}

\s Let $$\C_b = \{ X \in D^b(A) \mid \sup_{i \in \Z}\left(\ell(\Hom_{D^b(A)}(X, k[i])\right) < \infty \}. $$
 It can be easily verified that $\C_b$ is a thick subcategory of $D^b(A)$.

We need the following elementary fact. We give a proof for the convenience of the reader.
\begin{proposition}
\label{calc} Let $\{ \theta_n \}_{n \geq i_0} $ be positive real numbers such that \\  $\lim_{n \rt \infty} \theta_{n+1}/\theta_{n} = \alpha > 1$. Let $\{r_n\}_{n \geq i_0}$ be a sequence of non-negative real numbers such that $\limsup_{n\rt \infty}\sqrt[n]{r_n} < \alpha$. Then
\[
\lim_{n \rt \infty} \frac{r_n}{\theta_n} = 0.
\]
\end{proposition}
\begin{proof}
Set $c = \limsup_{n\rt \infty}\sqrt[n]{r_n}$.
  By \cite[3.36]{RW} we have $\lim_{n \rt \infty}\sqrt[n]{\theta_n} = \alpha$. Let $\epsilon > 0$ be such that $\alpha - \epsilon  > c + \epsilon$. So for $n \gg 0$
  we have $\theta_n \geq (\alpha-\epsilon)^n$ and $r_n \leq (c+\epsilon)^n$. It follows that
  $$\limsup \sqrt[n]{r_n/\theta_n} \leq \frac{c + \epsilon}{\alpha - \epsilon} < 1. $$
  Therefore the series $\sum_{n \geq i_0} r_n/\theta_n$ is convergent. The result follows.
\end{proof}
\s \label{xi-map-curv}
Let $1< \beta \leq \alpha$
 Let $\mmod_{\beta}(A)$ be the class of $A$-modules with curvature $< \beta$.   By an argument similar to \ref{closed}  we get that $\mmod_{ \beta}(A)$ is an extension closed subcategory of $\mmod(A)$.
 Let $\mmod_b(A)$ be the class of $A$-modules with bounded betti-numbers.   We note that $\mmod_b(A)$ is an extension closed subcategory of $\mmod(A)$.

 Let $\B$ be $\mmod_{\beta}(A)$ or $\mmod_b(A)$.
  We have a natural map
 $$G_0(\B) \xrightarrow{\xi} G_0(A).$$
 It is difficult to analyze $\xi$ directly.
\s \label{setup-short-curv} Let $\C = \C_\beta$ if $\B = \mmod_{\beta}(A)$ and let $\C = \C_b$ if $\B = \mmod_b(A)$.
  Let $\T$ be the Verdier quotient $D^b(A)/\C$.
 We have obvious functors $\B \rt \C$ and $\mmod(A) \rt D^b(A)$ by sending a module $M$ to its stalk complex.
 So we have a commutative diagram

 \[
  \xymatrix
{
  G_0(\B)
\ar@{->}[r]^{\xi}
\ar@{->}[d]
& G_0(A)
\ar@{->}[d]^{\theta}
& \
& \
\\
  G_0(\C)
\ar@{->}[r]^{\eta}
& G_0(D^b(A))
\ar@{->}[r]
& G_0(\T)
\ar@{->}[r]
&0
\
 }
\]
It is well-known that $\theta$ is an isomorphism. So if $\eta$ is not surjective then neither is $\xi$.

\s We have a group isomorphism $f \colon G_0(\mmod(A)) \rt \Z$ defined by sending $M$ to $\ell(M)$.  We have a group homomorphism $\delta  \colon G_0(D^b(A)) \rt \Z$ defined by  $\delta(\Yb) = \sum_{i\in \Z}(-1)^i\ell(H^i(\Yb))$. Note  in fact $\delta = f \circ\theta^{-1}$. So $\delta $ is an isomorphism.
Let $\Yb$ be a  bounded above complex of finitely generated $A$-modules. We do NOT assume $H^i(\Yb) = 0$ for $i \ll 0$. For $m \in \Z$ set
$$\chi_m(\Yb) = \sum_{j \geq 0}(-1)^j\ell(H^{m + j}(\Yb)). $$

Next we show
\begin{lemma}\label{eta-curv}(with hypotheses as in \ref{setup-short-curv}).
Suppose if possible $\eta$ is surjective. Then there exists a minimal  bounded above complex $\Xb \in K^{-,b}(\proj A)$ with
\begin{enumerate}[\rm (1)]
  \item $\chi(\Xb)  = \sum_{n \in \Z}(-1)^i\ell(H^i(\Xb))  = 0$.
  \item Set $\theta_n = \ell(\Hom_{D^b(A)}(\Xb, k[n]))$. Then $\lim_{n \rt \infty} \theta_{n+1}/\theta_n  = \alpha$.
 \item  Let $D \in \mmod_b(A)$. Then
   \begin{enumerate}[\rm (a)]
     \item The set $\{  |\chi_m(D\otimes \Xb)| \mid m \in \Z \} $  is bounded.
     \item The set $\{ \ell(H^{m}(D \otimes\Xb )) |  m \in \Z \}$ is bounded.
   \end{enumerate}
  \item  Let $D \in \mmod_{\beta}(A)$. Then
   \begin{enumerate}[\rm (a)]
     \item
     $$ \lim_{n \rt \infty} \frac{ |\chi_{-n}(D\otimes \Xb)|}{\theta_n} = 0.$$
     \item
     $$ \lim_{n \rt \infty} \frac{ \ell(H^{-n}(D\otimes \Xb))}{\theta_n} = 0.$$
   \end{enumerate}
\end{enumerate}
\end{lemma}
\begin{proof}
   Let $\T = D^b(A)/\C$. We have $G_0(\T) = 0$.
 Let
$$\A = \{ \Zb \mid \sum_{i \in \Z}(-1)^i\ell(H^i(\Zb)) = 0 \}.$$
Then $\A$ is a dense subcategory of $D^b(A)$ corresponding to the zero subgroup of $G_0(D^b(A))$.
Consider the full subcategory $\A^*$  of $\T$ whose objects are
$$\{ \Zb \mid \Zb \cong \Yb \ \text{in $\T$ for some $\Yb$ in $\A$} \}.$$
Then by Theorem \ref{dense}, $\A^*$ is a dense triangulated subcategory of $\T$. But $G_0(\T) = 0$. It follows that $\A^* = \T$. So there exists $\Xb \in \A$ such that we have an isomorphism $\phi \colon \Xb \cong \Pb$ in $\T$ where $\Pb$ is the minimal free resolution of $k$. We assume $\Xb$ is a minimal complex of free $A$-modules.
 We have a left fraction
\[
\xymatrix{
\
&\Zb
\ar@{->}[dl]_{u}
\ar@{->}[dr]^{f}
 \\
\Xb
\ar@{->}[rr]_{\phi = fu^{-1}}
&\
& \Pb
}
\]
where $\cone(u) \in \C$.  As $\phi$ is an isomorphism we also get $\cone(f) \in \C$.
Let \\  $r_n = \ell(\Hom_{D^b(A)}(\Zb, k[n]) $ and $\theta_n = \ell(\Hom_{D^b(A)}(\Xb, k[n]) $.

Claim 1: $\lim_{n \rt \infty} r_n/\beta_n(k) = 1$ and $\lim_{n\rt \infty} r_{n+1}/r_n = \alpha$

We have a triangle $t \colon \Zb \xrightarrow{f} \Pb \rt \Cb \rt \Zb[1]$ where $\Cb \in \C$. Set $c_n = \ell(\Hom_{D^b(A)}(\Cb, k[n]) $.
Taking $\Hom_{\D^b(A)}(-, k[n])$ we have an exact sequence
\[
\Hom(\Cb, k[n] ) \rt \Hom(\Pb, k[n] ) \rt \Hom(\Zb, k[n] ) \rt \Hom(\Cb, k[n + 1] )
\]
So we have
\begin{align*}
  \beta_n(k) &\leq r_n + c_n,\\
  r_n  &\leq \beta_n(k) + c_{n+1}.
\end{align*}
From the first estimate we have
$$1 -  c_n/\beta_n(k) \leq r_n/ \beta_n(k).$$
Note either $\{ c_n \}$ is bounded or $\Cb \in \C_\beta$. At any case we have  $\lim_{n \rt \infty} c_n/\beta_n(k) = 0$, see \ref{calc}.
Therefore  we have
$$ 1 \leq \liminf  r_n/ \beta_{n}(k)$$
From the second estimate we have
$$ r_n/\beta_{n}(k) \leq 1 + c_{n+1}/\beta_n(k).$$
By \ref{calc} we have
$$\limsup  r_n/\beta_n(k)  \leq 1.$$
So $\lim_{n \rt \infty} r_n/\beta_n(k) = 1$.
Finally note that
\[
\frac{r_{n+1}}{r_n}  = \frac{r_{n+1}}{\beta_{n+1}} \frac{\beta_{n+1}}{\beta_n} \frac{\beta_{n}}{r_n} \rt \alpha  \quad \text{as} \ n \rt \infty.
\]
Using this and a similar argument as in Claim-1 yields

Claim 2: $\lim_{n \rt \infty} \theta_n/r_n = 1$  and $\lim_{n \rt \infty} \theta_{n+1}/\theta_n = \alpha$.

Recall $\Xb$ is a minimal complex. Then $\theta_n = $ number of minimal generators of $\Xb^{-n}$.
We assume $H^i(\Xb) = 0$ for $i \leq m_0$.

3(a), 4(a):  Let $D \in \B$.
 Consider
\[
0 \rt \Omega^1(D) \rt A^{\mu(D)} \rt D \rt 0.
\]
Taking $-\otimes \Xb$ we get
\begin{equation*}
  0 \rt \Omega^1(D)\otimes\Xb \rt \Xb^{\mu(D)} \rt D\otimes \Xb \rt 0. \tag{$\dagger$}
\end{equation*}
Note $\Xb \in \ker \delta$. So $\sum_{i\in \Z} (-1)^i\ell(H^i(X)) = 0$. Let $m \leq m_0 $ and assume $m + c = m_0$. By ($\dagger$) we have
\[
\chi_m(D\otimes \Xb) = \chi_{m+1}(\Omega^1(D)\otimes \Xb) = \chi_{m+2}(\Omega^2(D)\otimes \Xb) = \cdots = \chi_{m_0}(\Omega^c(D)\otimes \Xb).
\]
3(a) Suppose $D \in \mmod_b(A)$. Say $\beta_n(D) \leq g$ for all $n \geq 0$. Then $e_0(\Omega^c(D)) \leq g\ell(A)$. We have
$$ |\chi_{m_0}(\Omega^c(D)\otimes \Xb)| \leq g\ell(A)(\sum_{n \leq -m_0}\theta_n). $$
Thus $|\chi_m(D\otimes \Xb)| $ is bounded.

4(a) Suppose $D \in \mmod_{\beta}(A)$. Then $e_0(\Omega^c(D)) \leq \beta_{-m + m_0}(D)\ell(A)$.
By \ref{calc} it follows that
$$ \frac{|\chi_m(D\otimes \Xb)| }{ \theta_{-m}} \leq \frac{\beta_{-m + m_0}(D)\ell(A)}{\theta_{-m}} (\sum_{n \leq -m_0}\theta_n) \rt 0 \quad \text{as} \ -m \rt \infty.$$

3(b), 4(b).
 By ($\dagger$) we have for $m \leq m_0$ and $m + c = m_0$ we have
$$ H^{m}(D\otimes \Xb) = H^{m + 1}(\Omega^1(D)\otimes\Xb) = \cdots = H^{m_0}(\Omega^c(D)\otimes \Xb).    $$
3(b) Suppose $D \in \mmod_b(A)$. Say $\beta_n(D) \leq g$ for all $n \geq 0$. We note that
$$ \ell(H^{m_0}(\Omega^c(D)\otimes \Xb)) \leq \theta_{-m_0}\ell(\Omega^c(D)) \leq \theta_{-m_0}g\ell(A). $$
The result follows.\\
4(b)Suppose $D \in \mmod_{\beta}(A)$. Then $e_0(\Omega^c(D)) \leq \beta_{-m + m_0}(D)\ell(A)$.
By \ref{calc} it follows that
$$ \frac{\ell(H^m(D\otimes \Xb)) }{ \theta_{-m}} \leq \frac{\beta_{-m + m_0}(D)\ell(A)}{\theta_{-m}}\theta_{-m_0} \rt 0 \quad \text{as} \ -m \rt \infty.$$
\end{proof}

Next we prove:
\begin{theorem}\label{eta-surj-curv} (with hypotheses as in \ref{setup-short-curv}). Further  assume there exists $D \in  \B$  with $\m^2D = 0$.  Then
\begin{enumerate}[\rm (1)]
  \item If $\alpha$ is irrational then $\eta$ is not surjective.
  \item If $\eta $ is surjective, let $(\alpha + 1)/\alpha = p/q$ where $p, q$ are co-prime positive integers. Then $p$ divides $\ell(D)$.
\end{enumerate}
\end{theorem}
\begin{proof}
   Suppose $\eta$ is surjective. We prove that necessarily $\alpha$ is rational.
Set $r(D) = \mu(\m D)$. We have an exact sequence
$$ 0 \rt k^{r(D)} = \m D \rt D \rt k^{\mu(D)} \rt 0.$$
Let $\Xb$ be as in Lemma \ref{eta-curv}. Taking $-\otimes \Xb$ we get
$$ 0 \rt \ov{\Xb}^{r(D)} \rt D \otimes \Xb \rt \ov{\Xb}^{\mu(D)} \rt 0$$
where $\ov{\Xb} = \Xb\otimes k$.

 Assume $H^i(\Xb) = 0$ for $i \leq m_0$.
Let $-n \leq m_0$. Then we have
$$ 0 \rt T_n \rt H^{-n}(\ov{\Xb})^{r(D)} \rt H^{-n}(D\otimes \Xb) \rt H^{-n}(\ov{\Xb})^{\mu(D)} \rt \cdots.$$
So we have
$$ \ell(T_n) + \chi_{-n}(D \otimes \Xb) = \ell(D)\chi_{-n}(\ov{\Xb}).$$
Thus we get
$$ \ell(T_n)/\theta_n + \chi_{-n}(D \otimes \Xb)/\theta_n = \ell(D) \left( 1 - (\sum_{j  \geq 0}(-1)^j\theta_{n-1 -j}/\theta_n) \right).$$
By \ref{eta-curv}3(a), 4(a) and \ref{lim}  we get
$$ \lim_{n \rt \infty} \ell(T_n)/\theta_n =  \ell(D) ( 1 - 1/(\alpha + 1)) = \ell(D) \alpha/(\alpha + 1).$$
We also have
$$ 0 \rt T_n \rt H^{-n}(\ov{\Xb})^{r(D)} \rt K_n \rt 0 \quad \text{where $K_n$ is a submodule of} \ H^{-n}(D\otimes \Xb).$$
So we have
$$ \ell(T_n)  + \ell(K_n) =\theta_nr(D).$$
By \ref{eta-curv}3(b), 4(b) it follows that $\lim_{n \rt \infty}\ell(K_n)/\theta_n = 0$. So
$$\lim_{n \rt \infty} \ell(T_n)/\theta_n = r(D).$$
Thus
$$r(D) = \ell(D) \alpha/(\alpha + 1).$$
Therefore
$$\ell(D)/r(D) = 1 + 1/\alpha.$$
It follows that $\alpha$ is rational. Furthermore if $1 + 1/\alpha = p/q$ where $p,q$ are positive  co-prime integers with $p > q$. Then it follows that $p$ divides $\ell(D)$. Thus we have shown both the assertions.
\end{proof}

\section{Modules with bounded betti numbers}
In this section we prove Theorem \ref{main}. We first show:
\begin{theorem}
\label{short-body} Let $(A,\m)$ be a short Artin local ring. Assume that $A$ is not a complete intersection and that there exists  a non-free $A$-module with bounded betti-numbers.   Then there exists
$c > 1$ such that $c$ divides $\ell(M)$ for every module $M$ with bounded betti-numbers.
\end{theorem}
\begin{proof}
  By \ref{short} we have $\ell(A) = 2h + 2$ and $\lim_{n \rt \infty}\beta_{n+1}^A(k)/\beta_n^A(k) = h$. So the techniques in section \ref{const} are applicable.
  Set $\B = \mmod_b(A)$ and $\C = \C_b$. \\
Claim : $\eta$  (as defined in \ref{setup-short-curv})  is NOT surjective.\\
   Suppose if possible $\eta$ is  surjective.
  Let $N \in \B$. Set $D  = \Omega^{1}_A(N)$. Then $D$ has bounded betti-numbers and $\m^2 D = 0$. Then by \ref{eta-surj-curv} it follows that $h+1$ divides $\ell(D)$. As $h+1$ divides $\ell(A)$ also, it follows that $h+1$ divides $\ell(N)$ also.

  Let $\Xb \in \C$. We assume $\Xb \in K^{b,-}(\proj A)$ and that (after shifting) $H^i(\Xb) = 0$ for $i \leq 1$. So $\Xb$ is quasi-isomorphic to a complex
  \[
 \Yb \colon  0 \rt M \rt F_{1} \rt F_{2} \rt \cdots \rt F_{s} \rt 0,
  \]
  where $F_i$ are free $A$-modules and $M$ is an $A$-module with bounded betti-numbers.
  Let $\delta \colon G_0(D^b(A)) \rt \Z$ be the isomorphism defined by \\ $\delta(\Cb)  = \sum_{i \in \Z}(-1)^i\ell(H^i(\Cb)).$
  We note that
  $$\delta(\Xb) =  \delta(\Yb) = \ell(M) - \ell(F_1) + \ell(F_2) + \cdots + (-1)^s\ell(F_s).$$
  We note that $\delta(\Xb) \in (h+1)\Z$. Thus $\eta$ is not surjective.
Therefore $\xi \colon G_0(\B) \rt G_0(A)$ is also not surjective. So there exists $c \geq 2$ such that $c$  divides $\ell(N)$ for every $N \in \B$.
\end{proof}
We now give
\begin{proof}[Proof of Theorem \ref{main}]
If $A$ is a complete intersection of multiplicity four then by \ref{ci} it follows that $2$ divides $\ell(M)$ for every module $M$ with bounded betti numbers.

Now assume that $A$ is not a complete intersection.

(1) If $A$ has \emph{no non-free $A$-module with bounded multiplicity} then clearly $\ell(A)$ divides lengths of  every free $A$-module.

(2) If $A$ has a non-free $A$-module with bounded multiplicity then the result holds by Theorem \ref{short-body}.
\end{proof}

\section{Flat extensions}\label{flat-section}
\s We say an Artin local ring $R$ has  property $\B$ if there exists $c \geq 2$ such that $c$ divides $\ell(M)$ for every $R$-module $M$ with bounded betti-numbers. By Theorem \ref{main}, short Artin local
rings satisfy $\B$. Theorem \ref{flat} shows that property $\B$ is preserved under certain flat extensions. We restate it for the convenience of the reader.
\begin{proposition}\label{flat-body}
Let $(R,\m) \rt (S,\n)$ be an extension of Artin local rings such that $S$ is a finite free $R$-module. Assume the induced extension of residue fields $R/\m \rt S/\n$ is an isomorphism. Then if $R$ has property $\B$ then so does $S$
\end{proposition}
\begin{proof}
We note that as $S$ is a finite $R$-module, any finitely generated $S$ module is also finitely generated as a $R$-module.
  Let $N$ be a $S$-module. By considering a composition series of $N$ it follows that $\ell_S(N) = \ell_R(N)$.

As $R$ has property $\B$, there exists $c_R \geq 2$ such that $c_A$ divides $\ell_R(M)$ for any $R$-module with bounded betti-numbers.

Now let $N$ be a $S$-module with bounded betti-numbers as a $S$-module. Let $\Pb$ be a minimal projective resolution of $N$ as an $S$-module. Then $\Pb$ is also a free resolution (need not be minimal) of $N$ as an $R$-module. It follows that the betti-numbers of $N$ as an $R$-module are bounded. So $c_R$ divides $\ell_R(N)$. But $\ell_S(N) = \ell_R(N)$. The result follows.
\end{proof}
The following  provides a large class of examples where Proposition \ref{flat-body} is applicable.
\begin{example}\label{example-1}
Let $R = k[X_1, \ldots, X_n]/I$ be a  Artin local ring with  maximal ideal $(X_1, \ldots, X_m)R$. Assume $R$ has property $\B$ (for example $R$ can be any short ring which is not a hypersurface). Let $T = k[Y_1, \ldots, Y_d]/(f_1, \ldots, f_d)$ where $f_1,\ldots, f_d$  is a graded $T$-regular sequence. Then $S = R\otimes_k T = R[Y_1,\ldots, Y_d]/(f_1, \ldots, f_d)$ is a finite free extension of $R$. Furthermore the residue field of $S$ is also $k$, the residue field of $R$.
\end{example}
The following example also yields another class of examples where Proposition \ref{flat-body} is applicable.
\begin{example}\label{ci-flat}
Let $A = k[X,Y]$ Let $f_1$ and $f_2$ be an $A$-regular sequence such that $\sqrt{(f_1, f_2)} = (X, Y)$. Assume there exists $g_1, g_2 \in B = k[U, V]$ such that
\begin{enumerate}
  \item $g_1, g_2$ are homogeneous quadratic regular sequence in  $B$.
  \item $\sqrt{(g_1, g_2)} = (U, V)$.
  \item $k[U, V]/(g_1, g_2)$ has multiplicity four.
  \item $f_i = g_i(X^m, Y^n)$ for fixed $m, n$.
\end{enumerate}
   .Note the extension $k[X^m, Y^n] \rt k[X, Y]$ is flat. So it induces a finite flat extension of Artin algebras $k[X^m, Y^n]/(f_1, f_2) \rt k[X,Y]/(f_1, f_2)$.
Notice \\ $k[X^m, Y^n]/(f_1, f_2) \cong R = k[U, V]/(g_1, g_2)$ is short Artin complete intersection of multiplicity four. Thus $k[X, Y]/(f_1, f_2)$ has property $\B$. Furthermore as $2$ divides $\ell_R(M)$ for any $R$-module $M$ with bounded betti-numbers, it follows that $\ell_S(N)$ is divisible by $2$ for any $S$-module $N$ with bounded betti-numbers.
\end{example}
A specific example where \ref{ci-flat} is applicable is the following:
\begin{example}\label{ci-specific}
Let $A = k[X, Y]/(X^{2m}, Y^{2n})$ for  some $m, n \geq 1$. Then if $M$ is any $A$-module with bounded betti-numbers then $\ell_A(M)$ is even.
\end{example}
We now give
\begin{proof}[Proof of Corollary \ref{ci-cor}]
This follows easily from \ref{example-1} and \ref{ci-specific}.
\end{proof}
\section{Proof of Theorem \ref{curv}}
In this section we give a proof of Theorem \ref{curv}.
We restate it for the convenience of the reader.
\begin{theorem}\label{curv-body}
 Let $(A,\m)$ be a short Artin local ring. Assume $A$ is not  a complete intersection  ring. There exists $b_A \geq 2$ such that if $M$ is any module with $\curv(M) < \curv(k)$ then $b_A$ divides $\ell(\Omega^i_A(M))$ for all $i \geq 1$.
 \end{theorem}
 \begin{proof}
 We may assume there exists a non-free $A$-module $M$ with $\curv(M) < \curv(k)$ for otherwise there is nothing to prove.
   By \cite[Theorem B]{L} if a short ring has a module $M$ with $\curv(M) < \curv(k)$ then its Poincare series is

 \[
P_A(z) = \frac{1}{1- dz + az^2} = \frac{1}{(1-r_1z)(1-r_2z)}.
\]
where $d = \ell(\m/\m^2)$, $a = \ell(\m^2)$, $r_1$ and $r_2$ are positive integers with $r_1 < r_2$ and $\curv(M) = r_1 $ and $\curv(k) = r_2$. Furthermore $a \geq d -1$.

We note that $\lim_{n \rt \infty} \beta_{n_+1}(k)/\beta_n(k)  = r_2 > 1$. So the techniques in section \ref{const} are applicable.
Let $\B = \mmod_\alpha(A)$ and $\C = \C_\alpha$. Let $\eta, \xi$ be as in \ref{setup-short-curv}.

(1) If $\eta$ is not surjective then $\xi$ is also not surjective. Then there exists $c \geq 2$ such that $c$ divides $\ell(M)$ for every $M \in \B$.

(2) Assume $\eta$ is surjective. Let $M\in \B$. Fix $i \geq 1$ and set $D = \Omega^i_A(M)$. Then $\m^2D = 0$. So by Theorem \ref{eta-surj-curv} it follows that $r_2 + 1$ divides $\ell(D)$.
 \end{proof}

\section{Higher dimensional rings}
In this section $(A,\m)$ is a \CM \ local ring of dimension $d  \geq 1$. For convenience we will assume that the residue field $k =A/\m$ of $A$ is infinite.
Set $h = \embdim(A) - d$ and $e_0(A)$ its multiplicity.
Let $\gr A = \bigoplus_{n \geq 0}\m^n/\m^{n+1}$ be the associated graded ring of $A$.

\s \label{high-setup}
Recall $A$ is said to have minimal multiplicity if  $e_0(A)  =  h +1$. We assume $A$ is not regular.
It is well-known that if $A$ has minimal multiplicity then  $\gr A$ is \CM, see \cite{S-min}. Let $g \in \m^2\setminus \m^3$ be such that the initial for $g^*$ of $f$ in $\gr A$ is a non-zero divisor.
Let
 $$\B_g =  \{ M \mid M \ \text{ is a MCM $A/(g)$-module such that $\projdim_A M <  \infty$}\}.$$
 Also set
 $$ \B^{(2)}  = \bigcup_{ \stackrel{\deg g^* = 2}{ g^* \text{nzd in} \ \gr A}}\B_g. $$
 We state Theorem \ref{min} here for the convenience of the reader.
 \begin{theorem}\label{min-body}
   (with hypotheses as in \ref{high-setup}). If $M \in \B^{(2)}$ then $ h + 1$ divides $e_0(M)$.
 \end{theorem}
 We need a few preliminaries to prove this result.
 \s \label{rc-body} Let $\Pb$ be the minimal projective resolution of $k$ as a $A$-module. Let $g \in \m^2\setminus \m^3$ be such that $g^*$ is $\gr A$-regular. Let $\ov{\Pb} = \Pb\otimes_A A/(g)$.
 We have an exact sequence
 \[
 0 \rt \Pb \xrightarrow{g} \Pb \rt \ov{\Pb} \rt 0.
 \]
 It follows that $H^{-n}(\ov{\Pb}) = 0$ for $n \ll 0$. Furthermore
   $$\chi(\ov{\Pb}) = \sum_{n \in \Z}(-1)^i\ell(H^i(\ov{\Pb})) = 0. $$
 Let $M \in \B_g$. Set $R = A/(g)$. Let $\bx = x_1, \ldots, x_{d-1}$ be maximal $M \oplus \Omega_R^1(M) \oplus R$-superficial sequence.
 Set $\Xb = \ov{\Pb} \otimes_{R} R/(\bx)$. Then one can show similarly $H^{-n}(\Xb) = 0$ for $n \ll 0$. Furthermore  $\chi(\Xb) = 0$.
 Set $B = R/(\bx)$. Note $\Xb \in K^{b,-}(\proj B)$ is a minimal complex and $\Hom_{D^b(B)}(\Xb, k[n]) = \beta_n$ where $\beta_n = \beta_n^A(k) $ since
 $\beta_n^A(k)$ is the minimal number of generators of $\Pb^{-n}$.
 \begin{lemma}\label{bella}
 (with hypotheses as \ref{rc-body}).
Set $D = \Omega_R^1(M)/\bx \Omega^R_1(M)$. Then
   \begin{enumerate}[\rm (a)]
     \item $\Omega^2_B(D) = D$.
     \item The set $\{  |\chi_m(D\otimes \Xb)| \mid m \in \Z \} $  is bounded.
     \item The set $\{ \ell(H^{m}(D \otimes\Xb )) |  m \in \Z \}$ is bounded.
     \item The Hilbert series of $B$ is $1 + (h  + 1)z + hz^2$. In particular $\m^3B  = 0$.
     \item $\m^2 D = 0$.
   \end{enumerate}
   \end{lemma}
   \begin{proof}
   (a). This follows from the fact that $\Omega^3_R(M) = \Omega^1_R(M)$.

   (b), (c) This follows from a same argument as in Lemma \ref{eta-curv}(3).

   (d) Note $h_A(z) = 1 + h z$. We have $g^*$ is $\gr A$-regular. So $h_{R)}(z) = (1+hz)(1+z) = 1 + (h+1) z + h z^2$, see \ref{reg-hilb}. We have $\gr R$ is \CM. So by \ref{sup-reg} it follows that $x_1^*,\cdots,  x_{d-1}^*$ is $\gr R$-regular. By \ref{reg-hilb} it follows that $h_B(z) = 1 + (h+1) z + h z^2$. The result follows as $B$ is Artin.

   (e) This follows as $D = \Omega^1_B(M/\bx M)$.
   \end{proof}
   We now give
   \begin{proof}[Proof of Theorem \ref{min-body}.]
   Let $g \in \m^2\setminus \m^3$ be such that the initial for $g^*$ of $f$ in $\gr A$ is a non-zero divisor.
Let $M \in \B_g$.
   We first consider the case when $h = 1$. Then $A$ is a hypersurface of multiplicity two.  Then $A/(g)$ is a complete intersection of multiplicity four. Let $\bx = x_1, \ldots, x_{d-1}$ be a maximal $A/(g)\oplus M$-superficial sequence. Then $e_0(M) = \ell(M/\bx M)$. The latter is even by Theorem \ref{ci}. Thus $2 = h + 1$ divides $e_0(M)$.

  Next  we consider the case when $h \geq 2$. Then $\lim_{ n \rt \infty} \beta_{n+1}/\beta_n = h > 1$. We do the construction as in \ref{rc-body} and \ref{bella}
  Note $\m^2 D = 0$, see \ref{bella}.
  Set $r(D) = \mu(\m D)$. We have an exact sequence
$$ 0 \rt k^{r(D)} = \m D \rt D \rt k^{\mu(D)} \rt 0.$$
Let $\Xb$ be as in Lemma \ref{bella}. Taking $-\otimes \Xb$ we get
$$ 0 \rt \ov{\Xb}^{r(D)} \rt D \otimes \Xb \rt \ov{\Xb}^{\mu(D)} \rt 0$$
where $\ov{\Xb} = \Xb\otimes k$.
Assume $H^i(\Xb) = 0$ for $i \leq m_0$.
Let $-n \leq m_0$. Then we have
$$ 0 \rt T_n \rt H^{-n}(\ov{\Xb})^{r(D)} \rt H^{-n}(D\otimes \Xb) \rt H^{-n}(\ov{\Xb})^{\mu(D)} \rt \cdots.$$
So we have
$$ \ell(T_n) + \chi_{-n}(D \otimes \Xb) = \ell(D)\chi_{-n}(\ov{\Xb}).$$
Thus we get
$$ \ell(T_n)/\beta_n + \chi_{-n}(D \otimes \Xb)/\beta_n = \ell(D) \left( 1 - (\sum_{j  \geq 0}(-1)^j\beta_{n-1 -j}/\beta_n) \right).$$
We have $|\chi_{-n}(D \otimes \Xb)|$ is bounded (see \ref{bella}). So by \ref{lim} we get
$$ \lim_{n \rt \infty} \ell(T_n)/\theta_n =  \ell(D) ( 1 - 1/(h + 1)) = \ell(D) h/(h + 1).$$
We also have
$$ 0 \rt T_n \rt H^{-n}(\ov{\Xb})^{r(D)} \rt K_n \rt 0 \quad \text{where $K_n$ is a submodule of} \ H^{-n}(D\otimes \Xb).$$
So we have
$$ \ell(T_n)  + \ell(K_n) =\beta_nr(D).$$
By \ref{bella},   it follows that $\ell(K_n)$ is bounded. Therefore
$$\lim_{n \rt \infty} \ell(T_n)/\beta_n = r(D).$$
Thus
$$r(D) = \ell(D) h/(h + 1).$$
Therefore
$$\ell(D)h = (h+1)r(D).$$
It follows that $h+1$ divides $\ell(D)$. But $\ell(D) = e_0(\Omega^1(M))$  As $ e_0(A) = 2(h+1)$, it follows that $h +1$ also divides $e_0(M)$.
   \end{proof}
\section{Ratios of betti numbers of certain Artin local rings}
In this section we compute $\lim_{n \rt \infty} \beta_{n + 1}(k)/\beta_n(k)$  of certain Artin local rings with residue field $k$. The results in this section is essentially known. We give proofs as we do not have a reference. If $(A,\m)$ is a Noetherian local ring then $P_A(z) = \sum_{n \geq 0}\beta_n^A(k)z^n$ denotes its Poincare series.

We begin with
\begin{lemma}\label{min-mult}
Let $(S,\n)$, $(T,\m)$ be rings with minimal multiplicity of codimension $h \geq 2$ and dimensions $0, 1$ respectively. Let $f \in \m^2$ be a non-zero divisor on $T$ and let $A = T/(f)$.
Then
\begin{enumerate}[\rm (1)]
  \item $\beta_n^S(k) = h^n$ for all $n \geq 0$.
  \item $P_S(z) = 1/(1-hz)$
  \item $P_T(z) = (1+z)P_S(z)$.
  \item $P_A(z)  =  P_T(z)/(1-z^2) = \frac{1}{(1-hz)(1-z)}$.
  \item $\lim_{n \rt \infty} \beta_{n+1}^A(k)/\beta_n^A(k)  = h$.
\end{enumerate}
\end{lemma}
\begin{proof}
(1) Note $\n = k^h$. It follows that $\beta_1(k) = h$ and $\beta_{n + 1} = h \beta_n$ for $n \geq 1$. The result follows.

(2) This follows from (1).

(3) After possibly going to a flat extension (to get infinite residue field) we may assume that there exists $x \in \m$ which is $T$-superficial. So $T/(x)$ has minimal multiplicity. The result follows from \cite[3.3.5]{A}.

(4) This follows from \cite[3.3.5]{A}.

(5) We have by (2), (3)
\[
\sum_{n \geq 0} \beta^A_n(k)z^n  = \frac{P_S(z)}{(1-z)}.
\]
Set $\beta_n = \beta^A_n(k)$. Then by above equality we obtain
\[
\beta_n = 1 + h + \cdots + h^n = \frac{h^{n+1} -1}{h -1}.
\]
So we obtain
\[
\frac{\beta_{n+1}}{\beta_n} = \frac{h^{n+1} -1}{h^{n} -1} = \frac{ h - \frac{1}{h^n} }{1 - \frac{1}{h^n}}.
\]
Taking limits as $n \rt \infty$ we obtain the result.
\end{proof}

\s  \label{short} Let $(A,\m)$ be a short Artin local ring (i.e., $\m^3 = 0$ and $\m^2 \neq 0$). If $A$ has a non-free module $M$ with bounded betti-numbers then by \cite[Theorem B]{L}
we have $\ell(\m/\m^2) = h + 1$, $\ell(\m^2) = h \geq 2$, socle of $A$ is $\m^2$ and
\[
P_A(z) = \frac{1}{1- (h+1)z + hz^2} = \frac{1}{(1-z)(1-hz)}.
\]
Than by an argument similar to \ref{min-mult}(5) we obtain $\lim_{n \rt \infty} \beta_{n+1}^A(k)/\beta_n^A(k)  = h$.

\section{Appendix}
In the appendix we calculate a limit which is crucial to us.
\s \label{lim-setup} Let $\{ \theta_n \}_{n \geq c} $ be a sequence such that $\lim_{n \rt \infty} \theta_{n + 1}/\theta_n = \xi > 1$.
 Set $\theta_n = 0$ for $n < c$. Also set
$$r_n  = \sum_{j \geq 0}(-1)^j\frac{\theta_{n-1-j}}{\theta_n}. $$
We show
\begin{lemma}\label{lim}
  (with hypotheses as in \ref{lim-setup}). We have
  $$ \lim_{n \rt \infty} r_n  = \frac{1}{\xi + 1}. $$
\end{lemma}
\begin{proof}
We note $\lim_{n \rt \infty} \theta_n = \infty$.
  Choose $\epsilon > 0$ such that $0 < \epsilon < 1/\xi$ and $(1/\xi) + \epsilon < 1$. We note that  $\lim_{n \rt \infty}\theta_n/\theta_{n+1} = 1/\xi$. Choose
  $m_0$ such that
  $$ \frac{1}{\xi} - \epsilon < \frac{\theta_{m-1}}{\theta_m} < \frac{1}{\xi} + \epsilon \quad \text{for all}\  m \geq i_0.$$
  We note for $j \geq 2$ we have
  $$ \frac{\theta_{m-j}}{\theta_m} = (\frac{\theta_{m-j}}{\theta_{m-j +1}}) (\frac{\theta_{m-j + 1}}{\theta_{m-j+2}} )\cdots ( \frac{\theta_{m-1}}{\theta_m}).$$
  So we have
  \begin{equation*}
    (\frac{1}{\xi} - \epsilon )^j < \frac{\theta_{m-j}}{\theta_m} < (\frac{1}{\xi} + \epsilon )^j. \tag{$\dagger$}
  \end{equation*}
  Let
  $$t_n = \sum_{n-1-j \leq i_0 -1}(-1)^j\frac{\theta_{n-1-j}}{\theta_n}.$$
  Then clearly $\lim_{n \rt \infty} t_n = 0$.
  Set
  $$ \alpha  = (\epsilon + 1/\xi) \quad \text{and} \quad \beta = ((1/\xi) - \epsilon). $$
  We have
  \begin{align*}
    r_n &= t_n + \sum_{n-1-j \geq i_0 }(-1)^j\frac{\theta_{n-1-j}}{\theta_n} \\
     &=  t_n + \sum_{n-1-j \geq i_0, j \ \text{even}}\frac{\theta_{n-1-j}}{\theta_n} \ - \ \sum_{n-1-j \geq i_0, j \ \text{odd} }\frac{\theta_{n-1-j}}{\theta_n} \\
     &\leq t_n + \alpha(1 + \alpha^2 + \cdots + \alpha^{2l}) - \beta^2(1+ \beta^2 + \cdots + \beta^{2s})
  \end{align*}
  where $l, s \rt \infty$ as $n \rt \infty$.
  So we have $r_n \leq t_n + u_n(\epsilon)$ where
  $$u_n(\epsilon) = \alpha \frac{1 - \alpha^{2l +2}}{1 - \alpha^2} - \beta^2 \frac{1-\beta^{2s + 2}}{1 - \beta^2}. $$
  We have that $ 0< \max \{\alpha, \beta \} < 1$. So
  $$u(\epsilon) = \lim_{n \rt \infty} u_n(\epsilon) =  \frac{\alpha}{1 - \alpha^2} - \frac{\beta^2}{1 - \beta^2}.$$
  We have
  $\limsup r_n \leq u(\epsilon)$. Taking $\epsilon \rt 0$ we obtain
  $$ \limsup r_n \leq \frac{1/\xi}{1 - (1/\xi)^2} - \frac{1/\xi^2}{1 -(1/\xi)^2} = \frac{1}{\xi + 1}.$$

We now compute $\liminf r_n$. We note as before
\begin{align*}
  r_n &= t_n + \sum_{n-1-j \geq i_0, j \ \text{even}}\frac{\theta_{n-1-j}}{\theta_n} \ - \ \sum_{n-1-j \geq i_0, j \ \text{odd} }\frac{\theta_{n-1-j}}{\theta_n} \\ \\
   &\geq t_n + \beta(1+ \beta^2 + \cdots + \beta^{2s}) - \alpha^2(1 + \alpha^2 + \cdots + \alpha^{2l})
\end{align*}
where $l, s \rt \infty$ as $n \rt \infty$.
  So we have $r_n \geq t_n + v_n(\epsilon)$ where
   $$v_n(\epsilon) = \beta \frac{1 - \beta^{2s +2}}{1 - \beta^2} - \alpha^2 \frac{1-\alpha^{2s + 2}}{1 - \alpha^2}. $$
  We have that $ 0< \max\{\alpha, \beta \} < 1$. So
   $$v(\epsilon) = \lim_{n \rt \infty} v_n(\epsilon) =  \frac{\beta}{1 - \beta^2} - \frac{\alpha^2}{1 - \alpha^2}.$$
  We have
  $\liminf r_n \geq v(\epsilon)$. Taking $\epsilon \rt 0$ we obtain
  $$ \liminf r_n \geq  \frac{1/\xi}{1 - (1/\xi)^2} - \frac{1/\xi^2}{1 -(1/\xi)^2} = \frac{1}{\xi + 1}.$$
 The result follows.
\end{proof}


\end{document}